\documentclass{amsart}
\usepackage{amsmath,amsthm,amsfonts,amssymb}
\usepackage{graphicx}
\usepackage[all]{xy}

\numberwithin{equation}{section}

\newtheorem{theorem}{Theorem}[section]
\newtheorem{definition}[theorem]{Definition}

\newtheorem{lemma}[theorem]{Lemma}
\newtheorem{remark}[theorem]{Remark}
\newtheorem{prop}[theorem]{Proposition}

\newtheorem*{nonothm}{Theorem}
\newtheorem*{nonocor}{Corollary}

\renewenvironment{proof}[1][\proofname. ]{ { \noindent \it #1}}{\qed \\}

\newcommand{\atl}{a_{t,\lambda}}
\newcommand{\btl}{b_{t,\lambda}}
\newcommand{\AI}{A_\infty}
\newcommand{\CG}{\mathcal G}
\newcommand{\LL}{\mathbb{L}}
\newcommand{\RR}{\mathbb{R}}
\newcommand{\CC}{\mathbb{C}}
\newcommand{\ZZ}{\mathbb{Z}}

\newcommand{\PP}{\mathbb{P}}
\newcommand{\bk}{\boldsymbol{k}}

\newcommand{\CB}{\mathcal{B}}
\newcommand{\CE}{\mathcal{E}}
\newcommand{\CF}{\mathcal{F}}
\newcommand{\CL}{\mathcal{L}}
\newcommand{\CM}{\mathcal{M}}
\newcommand{\Hom}{{\rm Hom}}

\newcommand{\locmir}{\mathcal{L}\mathcal{M}}
\newcommand{\calC}{\mathcal{C}}
\newcommand{\CO}{\mathcal O}

\newcommand{\id}{{\rm id}}

\newcommand{\CS}{\mathcal{S}}
\newcommand{\ind}{{\rm ind}}
\newcommand{\ltl}{{\LL_{t,\lambda}}}
\newcommand{\Ctl}{c_{t,\lambda}}
\newcommand{\ctl}{c_{t,\lambda}}
\newcommand{\CAtl}{A_{t,\lambda}}
\newcommand{\CBtl}{B_{t,\lambda}}

\begin{document}
\title{Noncommutative homological mirror symmetry of elliptic curves}
\author{Sangwook Lee}
\address{Korea Institute for Advanced Study\\ 85 Hoegi-ro Dongdaemun-gu, Seoul 02455, South Korea}
\email{swlee@kias.re.kr}

\begin{abstract}
We prove an equivalence of two $\AI$-functors, via Orlov's Landau-Ginzburg/Calabi-Yau correspondence. One is the Polishchuk-Zaslow's mirror symmetry functor of elliptic curves, and the other is a localized mirror functor from the Fukaya category of $T^2$ to a category of noncommutative matrix factorizations. As a corollary we prove that the noncommutative mirror functor $\locmir_{gr}^{\LL_t}$ realizes homological mirror symmetry for any $t$.
\end{abstract}
\maketitle
\smallskip
\noindent \textbf{Keywords :} Homological mirror symmetry, noncommutative mirror functor, elliptic curve.

\smallskip
\noindent \textbf{Mathematics Subject Classification 2010 :} 53D37.
\section{Introduction}
Since Kontsevich's Homological Mirror Symmetry conjecture, the understanding on various topics around it(such as Fukaya categorires, derived categories, Calabi-Yau manifolds, toric varieties and so on) has increased a lot. The conjecture itself was also proven for many cases, for example for any projective Calabi-Yau hypersurfaces(see \cite{PZ,AS,Seidelquar,She}). People have also tried to deal with cases beyond Calabi-Yau manifolds, by taking mirror partners as Landau-Ginzburg models.

There have been also attempts to achieve more conceptual understanding of homological mirror symmetry. One of such attempts is to construct homological mirror functors in more explicit way. Among works towards this question we follow the idea of Cho-Hong-Lau\cite{CHLau1,CHLau2}, who proposed a geometric way to construct homological mirror symmetry functors for some cases. The idea is to fix a weakly unobstructed reference Lagrangian $\LL$ which is able to generate the Fukaya category, and then take a matrix factorization of the difference of potential functions for $\LL$ and any weakly unobstructed Lagrangian.

In this paper we elaborate on the comparison of different kinds of mirror symmetries of elliptic curves, continuing the work in \cite{L}. In \cite{L} the author compared two mirror symmetries, one of which has the B-model as a derived category, and the other as a category of matrix factorizations of a commutative polynomial. The former is given by \cite{PZ,AS}, and the latter is due to \cite{CHLau1} which was mentioned above. On the other hand, their recent work \cite{CHLau2} gives a  two parameter family of mirror partners of elliptic curves, which come from deformations of the reference Lagrangian. A generic member of the family is now a category of matrix factorizations of a central element in a {\em noncommutative} ring. 

The noncommutativity arises due to the loss of symmetries among holomorphic triangles under the deformation, but the noncommutative potentials are elements of so-called {\em Sklyanin algebras} whose properties are relatively well known. We recall its definition here.
\begin{definition}\label{def:sklyanin}
Let $\Lambda$ be a base field.
\[ Sky(a,b,c):=\frac{\Lambda\langle x,y,z\rangle}{(axy+byx+cz^2,ayz+bzy+cx^2,azx+bxz+cy^2)}\] is called a {\em Sklyanin algebra} over $\Lambda$.
\end{definition}

We consider an element $W^\ltl$ of a Sklyanin algebra $Sky(\atl,\btl,\ctl)$ over the Novikov field $\Lambda$, for any $-1<t<1$ and $\lambda \in U(1)$:
\begin{align*} W^\ltl 
=&-\frac{\pi i q_0^{(3t+1)^2}\lambda^{\frac{1+3t}{2}}}{6}\Big(a'(u)(xyz+zxy+yzx)+b'(u)(zyx+xzy+yxz)\end{align*}
\[ +c'(u)(x^3+y^3+z^3)\Big),\]
\[ a(u):=q_0^{-(3t+1)^2}\lambda^{-\frac{1+3t}{2}}\atl,
\; b(u):=q_0^{-(3t+1)^2}\lambda^{-\frac{1+3t}{2}}\btl,\]
\[c(u):=q_0^{-(3t+1)^2}\lambda^{-\frac{1+3t}{2}}c_{t,\lambda}\]
where 
\[ \atl=\lambda^{\frac{1+3t}{2}}\big(\sum_{k\in \ZZ}\lambda^{3k}q_0^{(6k+1+3t)^2}\big),\;\btl=\lambda^{\frac{1+3t}{2}}\big(\sum_{k\in \ZZ}\lambda^{3k+2}q_0^{(6k+5+3t)^2}\big),\]
\[\ctl=\lambda^{\frac{1+3t}{2}}\big(\sum_{k\in \ZZ}\lambda^{3k+1}q_0^{(6k+3+3t)^2)}\big).\]
As its notation suggests, it is shown in \cite{CHLau2} that $W^\ltl$ is the noncommutative potential of the immersed Lagrangian $\ltl \subset \PP^1_{3,3,3}$ whose construction is in Section \ref{sec:ncfunctor}. By Artin-Tate-van den Bergh's result \cite{ATV}, the Sklyanin algebras $Sky(\atl,\btl,\ctl)$ are twisted homogeneous coordinate rings of an elliptic curve, with twisting automorphisms can be described explicitly. Then Artin-van den Bergh's result \cite{AV} enables us to think of the derived category of modules over twisted homogeneous coordinate ring modulo torsion as usual derived category of coherent sheaves on the elliptic curve.

Now we are ready to state our main theorem.

\begin{nonothm}[Theorem \ref{thm:main}]
Let $X$ be the mirror cubic ${\mathrm{Proj}}(R/W^{\LL_{0,-1}})$, where \[R=Sky(a_{0,-1},b_{0,-1},c_{0,-1})\cong\Lambda[x,y,z].\] Then we have the following commutative diagram of functors
\[\xymatrixrowsep{1pc}\xymatrix{
 D^\pi Fu(T^2) \ar[rr]^{\CS_i} \ar[dd]_{\mathrm{Polishchuk-Zaslow}}^{\widetilde{\CF}} & & D^\pi Fu(T^2)\ar[dd]^{\locmir_{gr}^{\ltl}} \\
 &  & \\
  D^b Coh(X) \ar[rr]_{\mathrm{Orlov}}^{\CG_i^{t,\lambda}} & &HMF_\ZZ(W^\ltl),}\]
  where
 \begin{itemize} 
  \item $D^\pi Fu(T^2)$ is the split-closed derived Fukaya category of $T^2$, i.e. the split-closure of the triangulated envelope of $Fu(T^2)$,
  \item $\widetilde{\CF}$ is the Polishchuk-Zaslow's mirror equivalence, 
  \item $\CG_i^{t,\lambda}$ is Orlov's LG/CY correspondence functor, 
  \item $\CS_i$ is induced by a symplectomorphism followed by a shift, 
  \item $\locmir_{gr}^{\ltl}$ is the localized mirror functor with respect to the Lagrangian $\ltl$.
  \end{itemize}

\end{nonothm}
We will review localized mirror functors $\locmir_{gr}^{\ltl}$ in Section \ref{sec:locmir}, and Orlov's LG/CY functors $\CG_i$ in Section \ref{sec:LGCY}. The autoequivalence $\CS_i$ will be constructed in the proof of Theorem \ref{thm:main}. According to the theorem, we have the following important result which was conjectured in \cite{CHLau2}.
\begin{nonocor}
The functor $\locmir_{gr}^{\LL_{t,\lambda}}$ is an equivalence for any $t$ and $\lambda$.
\end{nonocor}

\begin{remark}
In the author's previous work \cite{L}, the above commutative diagram was verified for the commutative potential $W^{\LL_{0,-1}}$, not for all $(t,\lambda)\neq (0,-1)$. We basically follow the strategy of \cite{L}, but also give technical verifications whether the previous arguments in the commutative mirror case can be applied to the noncommutative case.
\end{remark}

The organization of the paper is as follows. We begin with preliminaries on coherent bimodules which are noncommutative analogues of coherent sheaves. Then we recall $\AI$-categories, Fukaya categories and matrix factorizations. We also introduce Cho-Hong-Lau's localized mirror functor formalism, and then recall Orlov's LG/CY correspondence. Finally, we prove our main result.

\subsection*{Acknowledgements}
The author thanks Cheol-Hyun Cho for his interest and encouragement. He is grateful to Hansol Hong who raised the question which was the original motivation of this work. The author was supported by NRF-2007-0093859.

\section{Twisted homogeneous coordinate rings}
Recall the definition of a quasi-coherent sheaf $\CF$ on a variety $X$, that for any open set $U\subset X$ its sections over $U$ are given an $\CO(U)$-bimodule structure \[ \CO(U) \otimes \CF(U) \otimes \CO(U) \to \CF(U).\] 
We slightly change the right module action using an automorphism $\sigma$ of $X$ as follows. \begin{definition}[\cite{AV}] Let $\CF$ be a quasi-coherent sheaf on $X$. A {\em quasi-coherent $\CO$-bimodule} $\CF_\sigma$ from $\CF$ is an assignment of $\CO(U)$-$\CO(\sigma U)$-bimodule $\CF_\sigma(U)$ for each open set $U$
\[\CO(U) \otimes \CF_\sigma(U) \otimes \CO(\sigma U) \to \CF_\sigma(U)\]
where $\CF_\sigma(U):=\CF(U).$ In other words, if $s\in \CF(U)$ and $a \in \CO(\sigma U)$ then the right action is twisted as $sa=a^\sigma s$, where $a^\sigma:=\sigma^*a.$
\end{definition}
Given a coherent $\CO$-bimodule $\CF_\sigma$ and a coherent sheaf $\CG$, we define their tensor products in the most natural way as following:
\[(\CG \otimes \CF_\sigma)(U):=\CG(U) \otimes_{\CO(U)} \CF_\sigma(U),\]
\[(\CF_\sigma \otimes \CG)(U):=\CF_\sigma(U) \otimes_{\CO(\sigma U)} \CG(\sigma U).\]

We also define the tensor product $\CF_\sigma \otimes \CG_\sigma$ of two coherent $\CO$-bimodules by
\[(\CF_\sigma \otimes \CG_\sigma)(U):=\CF_\sigma (U) \otimes_{\CO(\sigma U)} \CG_\sigma (\sigma U).\]
Observe that $(\CF_\sigma \otimes \CG_\sigma)(U)$ is a $\CO(U)$-$\CO(\sigma^2 U)$-bimodule. In general, $(\CF^1_\sigma \otimes \cdots \otimes \CF^i_\sigma)(U)$ is a $\CO(U)$-$\CO(\sigma^{i}U)$-bimodule.

\begin{remark}
The definition of coherent $\CO$-bimodules in \cite{AV} is given in more general context. They start from sheaves on the product $X \times X$, and equip them with left and right module structures by pushforwards via left and right projections respectively. Then the bimodules we constructed above are just special cases, given by restrictions of pullback(with respect to left projections) sheaves to the graph of $\sigma$. Then the tensor products defined above are special cases of the general definition. Since we are only interested in those cases, we do not discuss the general setting further.
\end{remark}

\begin{definition}[\cite{AV}]
Let $\CB=\bigoplus_{n \in \ZZ}\CB_n$, where $\CB_n=\CL_\sigma^{\otimes n}$ and $\CL$ is a line bundle. A graded quasi-coherent sheaf(i.e. an $\CO[t,t^{-1}]$-module where $t$ is a formal Laurent variable) $\CM$ is called a {\em graded left $\CB$-module} if there is a multiplication rule
\[\CB_i \otimes \CM_j \to \CM_{i+j}\] 
which satisfies the module axiom.
More precisely on sections, it is given by
\[\CB_i(U) \otimes_{\CO(\sigma^i U)} \CM_j(\sigma^i U) \to \CM_{i+j}(U).\]
\end{definition}
In particular, if we are given a (usual nongraded) quasi-coherent sheaf $\CF$, then $\CB \otimes \CF$ under consideration of $\CF$ being a graded left module concentrated at degree 0 gives a new graded left $\CB$-module, whose graded pieces are given by $(\CB \otimes \CF)_n=\CB_n \otimes \CF$.

Now we recall the following:

\begin{prop}[\cite{AV}]
Let $\CF$ be a quasi-coherent sheaf, and $\CM$ be a graded $\CB$-module.
The functors
$\CF \mapsto \CB \otimes \CF$
and 
$\CM \mapsto \CM_0$
are quasi-inverses which define equivalences between the categories of quasi-coherent sheaves and of graded left $\CB$-modules.
\end{prop}

Let $\CL$ be a line bundle on $X$, and $\sigma$ an automorphism of $X$. $\CL$ is called {\em $\sigma$-ample} if for sufficiently large $m$, \[H^q(X,\CL_\sigma^{\otimes m})=0\] for any $q>0$. Observe that when $\sigma$ is the identity, $\sigma$-ampleness is the usual ampleness.

We are ready to state the following key fact:
\begin{theorem}[Theorem 1.3 of \cite{AV}]\label{thm:ncserre}
Let $\sigma$ be an automorphism of $X$ and $\CL$ be a $\sigma$-ample invertible sheaf on $X$. Let $B_\sigma$ be a noncommutative graded ring as follows:
\[ B_\sigma:=\bigoplus_{m \geq 0}H^0(X,\CL_\sigma^{\otimes m})\] where the ring structure is given by
\[a\cdot_\sigma b:=a\cdot (\sigma^n)^* b\] for $a\in B_n, b\in B_m$.
Then there is an equivalence of abelian categories
\[Coh(X) \simeq {\rm qgr-}B_\sigma:= {\rm gr-}B_\sigma /{\rm tors-}B_\sigma.\]
\end{theorem}

\begin{definition}[\cite{AV}]
The ring $B_\sigma$ is called a {\em twisted homogeneous coordinate ring} of $X$.
\end{definition}
The theorem enables us to think of coherent sheaves on $X$ as $B_\sigma$-modules. Given a coherent sheaf $\CE$, the corresponding $B_\sigma$-module is
\[ \Gamma_*(\CE):=\bigoplus_{n\in \ZZ} H^0(X,\CB_n \otimes \CE).\]
It is clear that
\[ \Gamma_*(\CO_X)=B_\sigma,\; \Gamma_*(\CO_X(1))=B_\sigma(1),\]
by considering the following isomorphisms
\begin{align*} 
\CB_n \otimes \CO&= \CO(1) \otimes \sigma^*\CO(1)\otimes \cdots\otimes (\sigma^{n-1})^* \CO(1) \otimes (\sigma^n)^*\CO \\
& \cong \CO(1) \otimes \sigma^*\CO(1)\otimes \cdots\otimes (\sigma^{n-1})^* \CO(1)
\end{align*}
and
\[ \CB_{n}\otimes \CO(1)=\CO(1)\otimes \sigma^*\CO(1)\otimes\cdots\otimes (\sigma^{n-1})^*\CO(1) \otimes (\sigma^n)^*\CO(1)=\CB_{n+1}.\]

\section{$\AI$-categories}
\subsection{Basic definitions}
\begin{definition}

The {\em Novikov field} is 
\[\Lambda:=\Big\{\sum_{i \geq 0} a_i T^{\lambda_i} \mid a_i \in \CC, \lambda_i \in \RR, \lambda_i \to \infty \;\mathrm{as} \; i \to \infty \Big\}.\]

\end{definition}

A filtration $F^\bullet\Lambda$ of $\Lambda$ is given by \[F^\lambda \Lambda:=\Big\{\sum_{i\geq 0}a_i T^{\lambda_i} \mid \lambda_i \geq \lambda {\rm \;for\; all\; }i \Big\} \subset \Lambda.\]

We write \[F^+\Lambda:=\Big\{\sum_{i\geq 0}a_i T^{\lambda_i} \mid \lambda_i > 0 {\rm \;for\; all\; }i \Big\}.\]

The {\em Novikov ring} $\Lambda_0$ is defined by $F^0 \Lambda$.
\begin{definition}
A {\em filtered $\AI$-category} $\calC$ over $\Lambda$ consists of a class of objects $Ob(\calC)$ and the set of morphisms $hom_{\calC}(A,B)$ for each pair of objects $A$ and $B$(from now on we just write $hom(A,B)$ for $hom_\calC(A,B)$ if there is no confusion), with the following conditions:

\begin{enumerate}
 \item $hom(A,B)$ is a filtered $\ZZ$-graded $\Lambda$-vector space for any $A,B\in Ob(\calC)$,
 \item for $k \geq 0$ there are multilinear maps of degree 1
  \[m_k: hom(A_0,A_1)[1] \otimes hom(A_1,A_2)[1] \otimes \cdots \otimes hom(A_{k-1},A_k)[1] \to hom(A_0,A_k)[1]\]
  such that they preserve the filtration and satisfy the {\em $\AI$-relation}
  \[\sum_{k_1+k_2=n+1}\sum_{i=1}^{k_1}(-1)^\epsilon m_{k_1}(x_1,...,x_{i-1},m_{k_2}(x_i,...,x_{i+k_2-1}),x_{i+k_2},...,x_n)=0\] where $\epsilon=\sum_{j=1}^{i-1}(|x_j|+1).$ 
   
\end{enumerate}

\end{definition}


\begin{definition}
For an object $A$ in an $\AI$-category, $e_A \in hom(A,A)$ is called a {\em unit} if it satisfies
 \begin{enumerate}
  \item $m_2(e_A,x)=x,\; m_2(y,e_A)=(-1)^{|y|}y$ for any $x \in hom(A,B)$, $y \in hom(B,A)$,
  \item $m_{k+1}(x_1,...,e_A,...,x_k)=0$ for any $k \neq 1$.
 \end{enumerate}
\end{definition}

\begin{definition}
An element $b \in F^+hom^1(A,A)$ is called a {\em weak bounding cochain} of $A$ if it is a solution of the weak Maurer-Cartan equation
\begin{equation}\label{eq:weakMC}
m(e^b):=m_0^A+m_1(b)+m_2(b,b)+ \cdots = PO(A,b)\cdot e_A\end{equation}
for some $PO(A,b) \in \Lambda.$ If such a solution exists, then $A$ is called {\em weakly unobstructed}. 
 If there exists a solution $b$ such that $PO(A,b)=0$, then $b$ is called a {\em bounding cochain} and $A$ is called {\em unobstructed}. $PO(A,b)$ is called the {\em Landau-Ginzburg superpotential} of $b$. \end{definition}
We denote $\mathcal{M}_{weak}(A)$ be the set of weak bounding cochains of $A$. Then $PO(A,\cdot)$ is a function on $\mathcal{M}_{weak}(A)$. We also define 
\[\CM_{weak}^\lambda(A):=\{b\in \CM_{weak}(A) \mid PO(A,b)=\lambda\}.\]

Given an $\AI$-category $\calC$, under the assumption $\CM^\lambda_{weak}(A)$ is nonempty for some objects, we define a new $\AI$-category $\calC_\lambda$ as
\[Ob(\calC_\lambda)=\bigcup_{A \in Ob(\calC)}\{A\} \times \CM_{weak}^\lambda(A),\]
\[hom_{\calC_\lambda}((A_1,b_1),(A_2,b_2))=hom_{\calC}(A_1,A_2)\]
with the following $\AI$-structure maps
\begin{align*}
m_k^{b_0,...,b_k}:&\;hom_{\calC_\lambda}((A_0,b_0),(A_1,b_1)) \otimes \cdots \otimes hom_{\calC_\lambda}((A_{k-1},b_{k-1}),((A_k,b_k)) \\
&\to hom_{\calC_\lambda}((A_0,b_0),(A_k,b_k)),\end{align*}
\[m_k^{b_0,...,b_k}(x_1,...,x_k):=\sum_{l_0,...,l_k}m_{k+l_0+\cdots+l_k}(b_0^{l_0},x_1,b_1^{l_1},...,b_{k-1}^{l_{k-1}},x_k,b_k^{l_k})\]
where $x_i \in hom_{\calC_\lambda}((A_i,b_i),(A_{i+1},b_{i+1})).$ $\AI$-relation is induced by the weak Maurer-Cartan equation (\ref{eq:weakMC}).

%
%
\begin{definition}
Let $\calC$ and $\calC'$ be $\AI$-categories. An {\em $\AI$-functor} between $\calC$ and $\calC'$ is a collection $\CF=\{\CF_i\}_{i \geq 0}$ consisting of

\begin{itemize}
 \item $\CF_0: Ob(\calC) \to Ob(\calC'),$
 \item $\CF_k: hom_{\calC}(A_0,A_1)[1] \otimes \cdots \otimes hom_{\calC}(A_{k-1},A_k)[1] \to hom_{\calC'}(\CF_0(A_0),\CF_0(A_k))[1]$ of degree $0$
\end{itemize}
which are subject to the following condition:
\begin{eqnarray*}
& \sum\limits_{i,j} (-1)^{|x_1|'+\cdots+|x_{i-1}|'} \CF_{i-j+k}(x_1,...,x_{i-1},m^\calC_{j-i+1}(x_i,...,x_j),x_{j+1},...,x_k)\\
=& \sum\limits_l m^{\calC'}_{l+1}(\CF_{i_1-1}(x_1,...,x_{i_1}),\CF_{i_2-i_1}(x_{i_1+1},...,x_{i_2}),...,\CF_{k-i_l}(x_{i_l+1},...,x_{k})).
\end{eqnarray*}

\end{definition}


Finally, we recall the definition of an {\em $\AI$-category over a noncommutative base} from \cite{CHLau2}. Let $\calC$ be an $\AI$-category over $\Lambda_0$ and $K$ be a noncommutative $\Lambda_0$-algebra. Then we define $\widetilde{\calC}:=K \hat{\otimes}_{\Lambda_0} \calC$, with the following $\AI$-operations
\[ m_k^{\widetilde{\calC}}(f_1 p_1,\cdots,f_k p_k):= f_k \cdots f_1\cdot m_k^\calC(p_1,\cdots,p_k)\]
where $f_1,\cdots,f_k$ are elements of $K$. It is clear that those operations satisfy the $\AI$-relation(see Lemma 2.7 of \cite{CHLau2}).

\section{Noncommutative mirror functor}\label{sec:locmir}
\subsection{Lagrangian Floer theory}
We briefly recall Lagrangian Floer theory, especially for Calabi-Yau manifolds. For more details, see \cite{FOOO} and also \cite{L}. 

Let $(M,\omega)$ be a symplectic manifold with a holomorphic volume form $\Omega$. Given two Lagrangians $L_0$ and $L_1$ which intersect transversely(if not, choose a Hamiltonian diffeomorphism $\psi$ so that $L_0$ and $\psi(L_1)$ are transverse), the morphism space $CF(L_0,L_1)$ is defined by $\bigoplus_{p\in L_0 \cap L_1}\Lambda\cdot p$. 
The underlying space does not depend on flat line bundles in our setting(if one wants to equip with higher rank vector bundles, then the definition must be modified).

\begin{definition}
An oriented Lagrangian submanifold in the CY manifold $(M,\omega,\Omega)$ is {\em graded} if there is a function $\theta_L:L \to \RR$ such that \[\displaystyle\frac{\Omega(X_1(p)\wedge \cdots \wedge X_n(p))}{|\Omega(X_1(p) \wedge \cdots \wedge X_n(p))|}=e^{i\theta_L(p)}\] for any positively oriented wedge of vector fields $X_1 \wedge \cdots \wedge X_n$ of $L$. $\theta_L$ is called the {\em phase function} of $L$. If $\theta_L$ is constant, then $L$ is called {\em special Lagrangian}. 
\end{definition}
$CF(L_0,L_1)$ is $\ZZ/2$-graded by {\em Maslov indices} of intersection points, but if $L_0$ and $L_1$ are graded, then we can equip $CF(L_0,L_1)$ with a $\ZZ$-grading. A main geometric feature of graded Lagrangian submanifolds is that they do not bound nontrivial holomorphic discs.


$\AI$-operations are constructed as follows. Let $p_i \in CF(L_{i-1},L_i)$ for $i=1,\cdots,k$ and $q \in CF(L_0,L_k).$ Define a moduli space $\CM(p_1,...,p_k;q)$ of $J$-holomorphic polygons whose domains are $D^2$ minus $k+1$ boundary points cyclically ordered by $z_1,...,z_k,z_0$, arcs between $p_i$ and $p_{i+1}$ are mapped inside $L_i$(and inside $L_k$ between $p_k$ and $q$), and the images near those punctures are asymptotically $p_1,...,p_k,q$, respectively. Let $\CM(p_1,...,p_k;q;\beta)$ be a subset of $\CM(p_1,...,p_k;q)$ which consists of holomorphic polygons of homotopy class $\beta$. Then the dimension of the moduli space is given by
\[\dim \CM(p_1,...,p_k;q;\beta)=k-2+\ind(\beta).\]
Fix a trivialization of $u^*TM$ so that we get paths of Lagrangian subspaces $l_0,l_1,...,l_k$ on $L_0,L_1,...,L_k$ respectively. Then we start from $T_{p_1}L_0$, at corners $p_i$ concatenate negative definite paths, move along $l_i$ until we arrive at $T_q L_k.$ At $q$ concatenate the positive definite path and move along $l_0$ to arrive at $T_{p_1} L_0$ again. The {\em index} of $u$ is defined by the winding number of the loop described above, and it depends only on the homotopy class of $u$. For $j=1,\cdots,k$, let $\nabla_j$ be the flat connection on $L_j$, and the tuple $(L_j,\nabla_j)$ be an object of the Fukaya category. 
We define 
\begin{align*}
m_k: &\;CF((L_0,\nabla_0),(L_1,\nabla_1)) \otimes \cdots \otimes CF((L_{k-1},\nabla_{k-1}),(L_k,\nabla_k)) \\
& \to CF((L_0,\nabla_0),(L_k,\nabla_k))\end{align*}
 by
\[m_k(p_1,...,p_k):=\sum_{\stackrel{p_i\in L_{i-1}\cap L_i,q\in L_0\cap L_k}{\ind[u]=2-k}}\# (\CM(p_1,...,p_k;q;[u]))T^{\omega(u)}\cdot Pal(\partial u)\cdot q\]
where $Pal(\partial u)\in U(1)$ means the holonomy around $\partial u$. We remark again that if we equip Lagrangians with higher rank local systems then the definition of $\AI$-operations should be more careful, but for rank 1 local systems the above definition is precise enough.

\begin{definition}
The {\em Fukaya category} of a symplectic manifold $(M,\omega)$ is an $\AI$-category $Fu(M,\omega)$ over $\Lambda$ whose objects are Lagrangian submanifolds equipped with flat unitary line bundles, with morphisms and their $\AI$-operations defined as above.
\end{definition}

\begin{remark}
In this paper we only consider the Fukaya category of the 2-torus, in which every object is unobstructed, so we denote the Fukaya category of $T^2$ by $Fu(T^2)$, instead of specifying the potential value $0$ as $Fu_0(T^2)$.
\end{remark}



\subsection{Noncommutative homological mirror functor of the elliptic curve}\label{sec:ncfunctor}
Given the elliptic curve $T^2=\CC/(\ZZ\oplus e^{2\pi i/3}\ZZ)$ with the natural $\ZZ/3$-action and for $t\in (-1,1)$, take a Lagrangian ${\LL}^{vert}_{t,\lambda}$ with nontrivial spin structure, and with a flat line bundle of holonomy $\lambda$(the meaning of $\lambda=e^{2\pi is} \in U(1)$ is that if a segment on it has reverse orientation with length $l$, then the parallel transport along the segment is given by multiplying $e^{2\pi is\cdot l/l_0}$ where $l_0$ is the total length of ${\LL}^{vert}_{t}$), which is the vertical line in the universal cover $\CC$ pointing upward through $(\frac{-t+1}{4},0)$. Then apply $\ZZ/3$-action $\tau$ to have two more Lagrangians and take the union 
\[\LL_{t,\lambda}:={\LL}^{vert}_{t,\lambda} \cup \tau({\LL}^{vert}_{t,\lambda}) \cup \tau^2({\LL}^{vert}_{t,\lambda}).\] 
(The reference Lagrangian of the previous section is the case $t=0$ and $\lambda=1$.) Define 
\[b:= x(X_1+X_2+X_3)+y(Y_1+Y_2+Y_3)+z(Z_1+Z_2+Z_3)\]
where $X_i$, $Y_i$ and $Z_i$($i=1,2,3$) are odd degree immersed sectors of $\ltl$(depicted in Figure \ref{immsec}).
\begin{figure}
\includegraphics[height=2in]{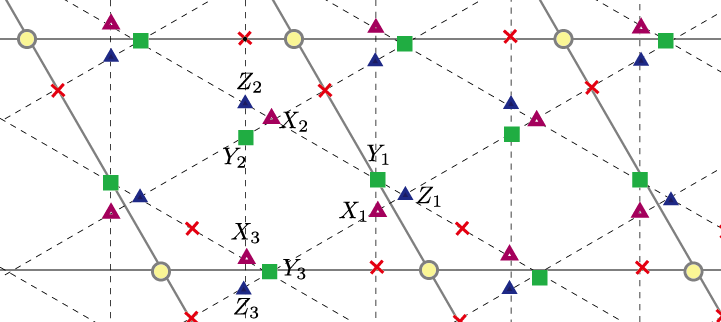}
\caption{Immersed generators of $\LL_{t,\lambda}$.}
\label{immsec}
\end{figure}

$x$, $y$, $z$ are formal variables of even degrees.
For $(t,\lambda)=(0,-1)$ we implicitly impose the commutativity relation on the free algebra $\Lambda\langle x,y,z\rangle$ to make $m_0^b$ a multiple of the unit. If we consider $m_0^b$ for general $\ltl$ without further relation on $\Lambda\langle x,y,z\rangle$, then it is not a mutiple of the unit, because there occur nonzero contributions of even degree immersed sectors. Therefore we need to add such contributions as relations on the generators. In \cite{CHLau2} the authors computed $m_0^b$ as follows, and so classified the relations which we need to quotient out. Let
\[ \bar{X}:=\bar{X}_1+\bar{X}_2+\bar{X}_3, \;\bar{Y}:=\bar{Y}_1+\bar{Y}_2+\bar{Y}_3,\; \bar{Z}:=\bar{Z}_1+\bar{Z}_2+\bar{Z}_3.\]

\begin{prop}[\cite{CHLau2}]
$m_0^b=W^\ltl \cdot 1_\ltl + h_{\bar{X}}\bar{X}+ h_{\bar{Y}}\bar{Y}+ h_{\bar{Z}}\bar{Z},$ where
\begin{align*} W^\ltl 
=&-\frac{\pi i q_0^{(3t+1)^2}\lambda^{\frac{1+3t}{2}}}{6}\Big(a'(u)(xyz+zxy+yzx)+b'(u)(zyx+xzy+yxz)\end{align*}
\[ +c'(u)(x^3+y^3+z^3)\Big),\]
\[ a(u):=q_0^{-(3t+1)^2}\lambda^{-\frac{1+3t}{2}}\atl,
\; b(u):=q_0^{-(3t+1)^2}\lambda^{-\frac{1+3t}{2}}\btl,\]
\[c(u):=q_0^{-(3t+1)^2}\lambda^{-\frac{1+3t}{2}}c_{t,\lambda}\]
where 
\[ \atl=\lambda^{\frac{1+3t}{2}}\big(\sum_{k\in \ZZ}\lambda^{3k}q_0^{(6k+1+3t)^2}\big),\;\btl=\lambda^{\frac{1+3t}{2}}\big(\sum_{k\in \ZZ}\lambda^{3k+2}q_0^{(6k+5+3t)^2}\big),\]
\[\ctl=\lambda^{\frac{1+3t}{2}}\big(\sum_{k\in \ZZ}\lambda^{3k+1}q_0^{(6k+3+3t)^2)}\big)\]
so that $a(u)$, $b(u)$ and $c(u)$ are holomorphic functions in $u$.

Weak Maurer-Cartan relations for general $(t,\lambda)$ are given by
\[ h_{\bar{X}}=h_{\bar{Y}}=h_{\bar{Z}}=0\] where
\[ h_{\bar{X}}=\atl yz+\btl zy+\ctl x^2, h_{\bar{Y}}=\atl zx+\btl xz+\ctl y^2, h_{\bar{Z}}=\atl xy+\btl yx+\ctl z^2. \]
The Landau-Ginzburg potential $W^{\ltl}$ is a central element of a Sklyanin algebra $\CAtl:=Sky(\atl,\btl,\Ctl)$.
\end{prop}

We define a correspondence between objects
\begin{align*}
 \CF_0: Fu(T^2) & \to MF(\CAtl,W^\ltl),\\
L & \mapsto \CAtl \otimes_{\Lambda_0}CF(\ltl,L).
\end{align*}
It can extend to an $\AI$-functor $\{\CF_i\}_{i \geq 0}$ exactly same as the commutative case. Only one thing is different: we need to define $\AI$-operations over the noncommutative base $\CAtl$ since the reference $\ltl$ has its weak bounding cochains in it. Also we need to be careful because the Floer complex is $\ZZ$-graded, and the corresponding matrix factorization is also $\ZZ$-graded. Therefore we need to specify the twisting of each summand of the corresponding module. We refer readers to Section 6.2 of \cite{L} for the details.

\section{LG/CY correspondence}\label{sec:LGCY}
A graded ring $A=\bigoplus_{d\geq 0}A_d$ is a {\em Gorenstein algebra with parameter $a$} if it has finite injective dimension $n$ and $\RR Hom_A(\bk,A) \cong \bk(a)[-n]$ where $\bk=A_0$. Let $D^b({\rm qgr-}A):=D^b({\rm gr-}A)/D^b({\rm tors-}A)$, and $D^b_{sg}(A):=D^b({\rm gr-}A)/{\rm Perf-}A$.

\begin{lemma}[\cite{Or}]
Let $\mathcal{U}_{\geq i}$ be the triangulated subcategory of $D^b({\rm gr-}A_{\geq i})$ generated by ${\bk}(e)$ with $e \leq -i$, and $\mathcal{P}_{\geq i}$ be the triangulated subcategory of $D^b({\rm gr-}A_{\geq i})$ generated by $A(e)$ with $e \leq -i$.
Then for any $i\in \ZZ$ we have semiorthogonal decompositions
\[D^b({\rm gr-}A_{\geq i})=\langle \mathcal{D}_i, \mathcal{U}_{\geq i} \rangle = \langle \mathcal{P}_{\geq i},\mathcal{T}_i \rangle\]
where $\mathcal{D}_i$ is equivalent to $D^b({\rm qgr-}A)$ and $\mathcal{T}_i$ is equivalent to $D^{gr}_{sg}(A)$.
\end{lemma}
We have an explicit equivalence from $D^b({\rm qgr-}A)$ to $\mathcal{D}_i$ by 
\[\RR \omega_i: D^b({\rm qgr-}A) \to D^b({\rm gr-}A_{\geq i}),\]
\[\RR \omega_i(M):=\bigoplus_{k=i}^\infty \RR\Hom_{D^b({\rm qgr-}A)}(\pi A, M(k))\]
where $\pi: D^b({\rm gr-}A) \to D^b({\rm qgr-}A)$ is the projection.

\begin{theorem}[\cite{Or}]
Let $A$ be a Gorenstein algebra with parameter $a$. Then $D^{gr}_{sg}(A)$ and $D^b({\rm qgr-}A)$ are related as following:
 \begin{enumerate}
  \item if $a>0$, for each $i \in \ZZ$ there are fully faithful functors $\Phi_i : D^{gr}_{sg}(A) \to D^b({\rm qgr-}A)$ and semiorthogonal decomposition
   \[D^b({\rm qgr-}A)=\langle \pi A(-i-a+1),...,\pi A(-i),\Phi_i D^{gr}_{sg}(A)\rangle\] where $\pi: D^b({\rm gr-}A) \to D^b({\rm qgr-}A)$ is the natural projection and 
   \[\Phi_i:D^{gr}_{sg}(A) \simeq \mathcal{T}_i \hookrightarrow D^b({\rm gr-}A) \stackrel{\pi}{\rightarrow} D^b({\rm qgr-}A).\]
  \item if $a<0$, for each $i \in \ZZ$ there are fully faithful functors $\CG_i: D^b({\rm qgr-}A) \to D^{gr}_{sg}(A)$ and semiorthogonal decomposition 
   \[D^{gr}_{sg}(A)=\langle q {\bf k}(-i),...,q {\bf k}(-i+a+1),\CG_i D^b({\rm qgr-}A)\rangle\] where $q: D^b({\rm gr-}A) \to D^{gr}_{sg}(A)$ is the projection and 
   \[\CG_i: D^b({\rm qgr-}A) \simeq \mathcal{D}_{i-a} \hookrightarrow D^b({\rm gr-}A) \stackrel{q}{\rightarrow} D^{gr}_{sg}(A).\]
  \item if $a=0$, $D^{gr}_{sg}(A)$ and $D^b({\rm qgr-}A)$ are equivalent via $\Phi_i$ and $\CG_i$.
 \end{enumerate}
\end{theorem}

\section{Main theorem}\label{sec:main}
Given two Lagrangians $L_{(1,0)}$ and $L_{(1,-3)}$, we compare objects and morphisms via functors $\CG^{t,\lambda}_i \circ \CF$ and $\locmir^{\LL_{t,\lambda}}\circ\CS_i$. Let $i=0$, and recall the notation $\CAtl=Sky(\atl,\btl,\ctl)$, $\CBtl=\CAtl/W^\ltl$. We recall an important result.
\begin{theorem}[\cite{ATV}]
The algebra $\CBtl$ is isomorphic to a twisted homogeneous coordinate ring $B_{\sigma}$ of the elliptic curve $X={\rm Proj}(\Lambda[x,y,z]/W^{\LL_{0,-1}})$ with some automorphism $\sigma:X\to X$. 
\end{theorem}
Under the isomorphism $\CBtl \cong B_\sigma$ and the equivalence
$D^b Coh(X) \simeq D^b({\rm qgr-}B_{\sigma}),$ we have
\[ D^b Coh(X) \simeq D^b({\rm qgr-}\CBtl)\] 
and the objects $\CO$ and $\CO(1)$ correspond to $\pi B_{t,\lambda}$ and $\pi B_{t,\lambda}(1)$ in $D^b({\rm qgr-}\CBtl)$ respectively($\pi$ is the projection ${\rm gr-}\CBtl \to {\rm qgr-}\CBtl={\rm gr-}\CBtl/{\rm tors-}\CBtl$). Orlov's result can be applied to the ring $\CBtl:=\CAtl/W^\ltl$ since it is a Gorenstein algebra(of parameter $0$). Then we have
\[
\RR\omega_0(\pi B_{t,\lambda})=\bigoplus_{k=0}^\infty \RR \Hom_{D^b({\rm qgr-}B_{t,\lambda})}(\pi B_{t,\lambda},\pi B_{t,\lambda}(k)), \]
\[
 \RR\omega_0(\pi B_{t,\lambda}(1))=\bigoplus_{k=0}^\infty \RR \Hom_{D^b({\rm qgr-}B_{t,\lambda})}(\pi B_{t,\lambda},\pi B_{t,\lambda}(k+1)) \]
 which are both objects in $D^b({\rm gr-}\CBtl)$, and we can compute them exactly same as commutative case, i.e.
\[ \RR\omega_0(\pi \CBtl(1)) \simeq \CBtl(1)_{\geq 0},\] 
\[\RR\omega_0(\pi \CBtl)\simeq {\rm Cone}(\phi:\bk[-2] \to \CBtl)\]
from the fact that $\CBtl$ is a $2$-dimensional Gorenstein algebra with parameter $0$.


We recall another important algebraic property of Sklyanin algebras.
\begin{theorem}[\cite{IS}]
Any Sklyanin algebra is Koszul, i.e. the module $\bk$ admits a linear free resolution.
\end{theorem}

Now, we state and prove our main theorem.
\begin{theorem}\label{thm:main}
Let $X$ be the mirror cubic ${\mathrm{Proj}}(R/W^{\LL_{0,-1}})$, where \[R=Sky(a_{0,-1},b_{0,-1},c_{0,-1})\cong\Lambda[x,y,z].\] Then we have the following commutative diagram of functors
\[\xymatrixrowsep{1pc}\xymatrix{
 D^\pi Fu(T^2) \ar[rr]^{\CS_i} \ar[dd]_{\mathrm{Polishchuk-Zaslow}}^{\widetilde{\CF}} & & D^\pi Fu(T^2)\ar[dd]^{\locmir_{gr}^{\ltl}} \\
 &  & \\
  D^b Coh(X) \ar[rr]_{\mathrm{Orlov}}^{\CG_i} & &HMF_\ZZ(W^\ltl),}\]
  where $\widetilde{\CF}$ is the Polishchuk-Zaslow's mirror equivalence, $\CG_i$ is Orlov's LG/CY correspondence functor, $\CS_i$ is induced by a symplectomorphism followed by a shift, and $\locmir_{gr}^{\ltl}$ is the localized mirror functor with respect to the Lagrangian $\ltl$.

\end{theorem}
\begin{proof}
We can find the linear free resolution of $\bk=\CAtl/(x,y,z)$ over $\CAtl$ explicitly as follows:
\[\xymatrix{
0 \ar[r] & \CAtl(-3) \ar[r]^{{\left(\begin{smallmatrix}-x \\ -y \\ -z\end{smallmatrix}\right)}} & \CAtl(-2)^3 \ar[rr]^{\tiny\left(\begin{smallmatrix}\ctl x & \btl z & \atl y \\ \atl z & \ctl y & \btl x \\ \btl y & \atl x & \ctl z\end{smallmatrix}\right)}
& &  \CAtl(-1)^3 \ar[rr]^-{\left(\begin{smallmatrix}-x & -y & -z\end{smallmatrix}\right)} & & \CAtl\ar[r] & 0.}\]
The choice of signs of entries $-x$, $-y$, $-z$ instead of $x$, $y$ and $z$ is just irrelevant, but we made such a choice so that later we can compare objects and morphisms clearer. We denote the differential maps in total by $s_0$. The endomorphism $W^\ltl\cdot 1$ of the complex is homotopic to zero because $W^{\ltl}$ is in the ideal $(x,y,z)$, so we have a homotopy from $W^\ltl\cdot 1$ to the zero map by
\begin{equation}\label{whomotopy}
\xymatrix{
0 \ar[r] & \CAtl(-6) \ar[r]^{s_0} \ar[d]_{W^\ltl} & \CAtl(-5)^3 \ar[r]^{s_0} \ar[ld]^{s_1} \ar[d]^{W^\ltl}
& \CAtl(-4)^3 \ar[r]^-{s_0} \ar[d]^{W^\ltl}\ar[ld]^{s_1}  & \CAtl(-3)\ar[r] \ar[ld]^{s_1}\ar[d]^{W^\ltl} & 0 \\
0 \ar[r] & \CAtl(-3) \ar[r]_{s_0} & \CAtl(-2)^3 \ar[r]_{s_0}
&  \CAtl(-1)^3 \ar[r]_-{s_0} &  \CAtl\ar[r] & 0.
}
\end{equation}

Now we apply Dyckerhoff's spectral sequence argument(see Lemma 2.6 of \cite{Dyc} and its proof) to our situation and observe that the following double complex is a resolution of $\bk$ over $\CBtl$:
\[\xymatrix{
\ar[d] & \ar[d] \ddots \ar[d] & \ar[d] & & & \\
\CBtl(-6) \ar[r]^{s_0} & \CBtl(-5)^3 \ar[r]^{s_0} \ar[d]^{s_1} & \CBtl(-4)^3\ar[r]^{s_0} \ar[d]^{s_1} & 
\CBtl(-3) \ar[r] \ar[d]^{s_1} & 0 & \\
& \CBtl(-3) \ar[r]^{s_0} & \CBtl(-2)^3 \ar[r]^{s_0} & \CBtl(-1)^3 \ar[r]^-{s_0} & \CBtl \ar[r] & 0.}\]
There is no obstacle to apply Dyckerhoff's argument to construct the resolution of $\bk$ over a  noncommutative algebra, because the main feature of the argument is based on that $W^\ltl\cdot 1$ is nullhomotopic. Only difficulty in the noncommutative setting is that we cannot easily construct $s_1$ purely by an algebraic argument(i.e. solving homotopy equations directly) as we could do much easier for the commutative case. But the only requirement for $s_1$ is that the map $s_0 + s_1$ constructs a matrix factorization of $W^\ltl$, and this feature of $s_1$ is specific enough as we will see soon, so we do not try to compute $s_1$ explicitly.

From the above diagram, the resolution of $\bk[-2]$ over $\CBtl$ is given by
\[\xymatrixcolsep{2pc}\xymatrix{
 \cdots \ar[r] & {\CBtl(-6)\oplus  \CBtl(-5)^3}\ar[rrrr]^{\left(\begin{smallmatrix} 0& *& * & * \\ * &  \ctl x & \btl z & \atl y \\ * & \atl z & \ctl y & \btl x \\ *& \btl y &  \atl x & \ctl z\end{smallmatrix}\right)} 
 & & & &\CBtl(-3) \oplus \CBtl(-4)^3 }\]
 \[\xymatrixrowsep{0.2pc}\xymatrixcolsep{3.5pc}\xymatrix{
 \ar[r]^-{\left(\begin{smallmatrix}0 & -x & -y & -z \\-x & * & *& *
 \\-y & *& * &* \\-z &* & *& *\end{smallmatrix}\right)} &
  \CBtl(-3)\oplus \CBtl(-2)^3 \ar[rr]^-{\left(\begin{smallmatrix} * &  \ctl x & \btl z & \atl y \\ * & \atl z & \ctl y & \btl x \\ * & \btl y &  \atl x & \ctl z\end{smallmatrix}\right)} & & \CBtl(-1)^3 \\
  &  \fbox{{\rm 0th}} & & \fbox{{\rm 1st}} \\
  \ar[r]^-{\left(\begin{smallmatrix} -x&-y&-z\end{smallmatrix}\right)}  & \CBtl \ar[r] & 0  \\ 
 &  \fbox{{\rm 2nd}} & &}
\]
and a map $\bk[-2] \to \CBtl$ is given by $ \phi: \CBtl(-3)\oplus \CBtl(-2)^3 \to \CBtl $
with \[ \phi \circ \left(\begin{matrix}0 & -x & -y & -z \\-x & * & *& *
 \\-y & *& * &* \\-z &* & *& *\end{matrix}\right)=0.\] 
 
 A homotopically nontrivial map $\phi$ is given by $\left(\begin{matrix}0 & Q_{12} & Q_{13} & Q_{14}\end{matrix}\right)$ which is the first row of the differential $\CBtl(-6)\oplus \CBtl(-5)^3 \to \CBtl(-3) \oplus \CBtl(-4)^3.$ Any other row gives a chain map but it is nullhomotopic. Hence, we get ${\rm Cone}(\phi:\bk[-2]\to \CBtl)=\RR \omega_0(\pi \CBtl)$ by
 
\[\xymatrixrowsep{0.2pc}\xymatrixcolsep{4.5pc}\xymatrix{
\cdots \ar[r]^-{\left(\begin{smallmatrix}0 & -x & -y & -z \\-x & * & *& *
 \\-y & *& * &* \\-z &* & *& *\end{smallmatrix}\right)}  & \CBtl(-3)\oplus \CBtl(-2)^3 \ar[rr]^{\left(\begin{smallmatrix} 0& *& * & * \\ * &  \ctl x & \btl z & \atl y \\ * & \atl z & \ctl y & \btl x \\ *& \btl y &  \atl x & \ctl z\end{smallmatrix}\right)} 
 & &
 \CBtl \oplus \CBtl(-1)^3 \\ & & &\fbox{{\rm 0th}} }\]
\[\xymatrixrowsep{0.2pc}\xymatrixcolsep{4pc}\xymatrix{ \ar[r]^-{\left(\begin{smallmatrix} 0&-x&-y&-z\end{smallmatrix}\right)}
 &\CBtl  \ar[r] & 0. \\
   &  \fbox{{\rm 1st}} & & 
 }\]

Since $\RR^i\omega_0 (\pi \CBtl(1))$ is nonzero only for $i=0$, $\RR\omega_0(\pi\CBtl(1))$ is isomorphic to the resolution of $\RR^0\omega_0(\pi\CBtl(1))\cong \CBtl(1)_{\geq 0}$, and the resolution is 
\[\xymatrixcolsep{2.6pc}\xymatrix{
\ar[r] & \CBtl(-2) \oplus \CBtl(-3)^3 \ar[rr]^-{\left(\begin{smallmatrix}0 & -x & -y & -z \\-x & * & *& *
 \\-y & *& * &* \\-z &* & *& *\end{smallmatrix}\right)} & & \CBtl(-2)\oplus \CBtl(-1)^3}\]
 \[\xymatrixcolsep{2.6pc}\xymatrix{ \ar[rrr]^-{\left(\begin{smallmatrix}* &  \ctl x & \btl z & \atl y \\ * & \atl z & \ctl y & \btl x \\ * & \btl y &  \atl x & \ctl z\end{smallmatrix}\right)} & & & \CBtl^3 \ar[r]
 & 0}\]
whose augmentation map is 
\[ \left(\begin{array}{cccc}x& y & z\end{array}\right):\CBtl^3\to {\CBtl}(1)_{\geq 0}.\]

Now we apply Polishchuk-Zaslow's mirror functor to Lagrangians $L_{(1,0)}$ and $L_{(1,-3)}$ with trivial holonomies, so that we obtain $\CO_X$ and $\CO_X(1)$. There are three morphisms(i.e. intersections) from $L_{(1,0)}$ to $L_{(1,-3)}$, and the mirror correspondence of \cite{PZ} is given by
\[ (0,0):L_{(1,0)} \to L_{(1,-3)} \longleftrightarrow y: \CO_X\to \CO_X(1),\]
\[ \big(0,\frac{1}{3}\big):L_{(1,0)} \to L_{(1,-3)} \longleftrightarrow x: \CO_X\to \CO_X(1),\]
\[ \big(0,\frac{2}{3}\big):L_{(1,0)} \to L_{(1,-3)} \longleftrightarrow z: \CO_X\to \CO_X(1).\]
Applying Orlov's functor, the maps $x,y,z: \CO_X \to \CO_X(1)$ correspond to $\phi_x,\phi_y,\phi_z: \RR\omega_0(\pi \CBtl) \to \RR\omega_0(\pi \CBtl(1))$, which induce $x,y,z: \CBtl \to \CBtl(1)_{\geq 0}$ on 0th cohomologies. $\phi_x$, $\phi_y$ and $\phi_z$ are determined by maps 
\[\phi_x^0,\;\phi_y^0,\;\phi_z^0: \CBtl\oplus \CBtl(-1)^3 \to \CBtl^3\]
and for them to induce $x$, $y$ and $z$ from $\CBtl$ to $\CBtl(1)_{\geq 0}$, they are expressed as
\[ \phi_x^0=\left(\begin{array}{cccc}1 & * & * & * \\ 0 & * & * & * \\ 0 & * & * & *\end{array}\right),\;\phi_y^0=\left(\begin{array}{cccc}0 & * & * & * \\ 1 & * & * & * \\ 0 & * & * & *\end{array}\right),\;
\phi_z^0=\left(\begin{array}{cccc}0 & * & * & * \\ 0 & * & * & * \\ 1 & * & * & *\end{array}\right).\] 

Similarly, for degree 1 morphisms we have the mirror correspondence
\[ (0,0):L_{(1,-3)} \to L_{(1,0)} \longleftrightarrow y^*: \CO_X(1)\to \CO_X[1],\]
\[ \big(0,\frac{1}{3}\big):L_{(1,-3)} \to L_{(1,0)} \longleftrightarrow x^*: \CO_X(1)\to \CO_X[1],\]
\[ \big(0,\frac{2}{3}\big):L_{(1,-3)} \to L_{(1,0)} \longleftrightarrow z^*: \CO_X(1)\to \CO_X[1]\]
and maps from $\RR\omega_0(\CO_X(1)) \to \RR\omega_0(\CO_X[1])$ corresponding to $x^*$, $y^*$ and $z^*$ are determined by maps $\psi_x^0,\;\psi_y^0,\;\psi_z^0\;:\CBtl^3 \to \CBtl$. The duality between $\{x,y,z\}$ and $\{x^*,y^*,z^*\}$ is achieved only if we define
\[ \psi_x^0=\left(\begin{array}{ccc}1 & 0 & 0\end{array}\right),\; \psi_y^0=\left(\begin{array}{ccc}0 & 1 & 0\end{array}\right),\; \psi_z^0=\left(\begin{array}{ccc}0 & 0 & 1\end{array}\right).\]

Now, let $\CS_0$ be a symplectomorphism of $T^2$ given by an affine transformation followed by a translation, such that it maps $L_{(1,0)}$ to $\LL^{vert}_{t,\lambda}$ and $L_{(1,-3)}$ to $\tau(\LL^{vert}_{t,\lambda})$. If we send $L_{(1,0)}$ and $L_{(1,-3)}$ by $\locmir^{\LL_{t,\lambda}} \circ\CS_0$, then we get two $4 \times 4$ matrix factorizations. Let $L_s=(\phi_s\circ \CS_0)(L_{(1,0)})$ where $\phi_s$ is a small Hamiltonian perturbation, and $L=\CS_0(L_{(1,0)}).$ Similarly, let $L':=\CS_0(L_{(1,-3)}).$ See Figure \ref{44mf}. 
\begin{figure}
\includegraphics[height=3in]{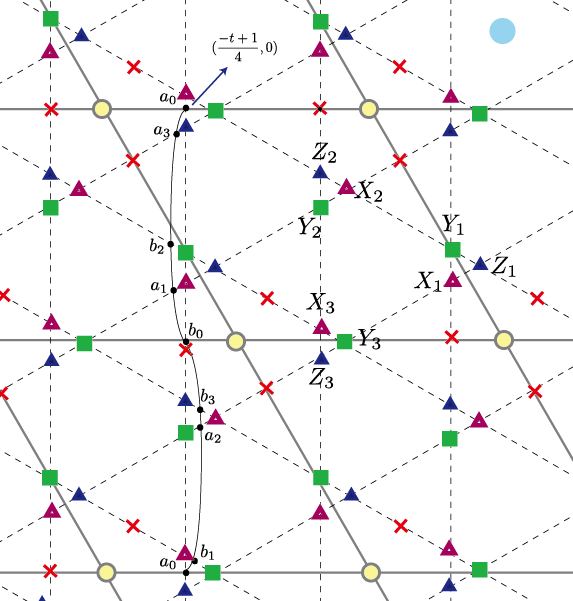}
\caption{$CF(\ltl,L_s)$ generated by even intersections $a_0,\cdots,a_3$ and odd generators $b_0,\cdots,b_3$.}
\label{44mf}
\end{figure}

Linear entries of their differentials are rather easily computed as follows. The strips which become area 0 after taking the limit $\phi_s \to \id$ are those $a_0 \mapsto b_i$ and $a_i \mapsto b_0$ for $i=1,\cdots,3$. They all have negative signs. Other strips(of nonzero areas) which give linear entries are identified with polygons which give terms of Maurer-Cartan relations. For example, let us consider strips from $b_1$ to $a_1,a_2,a_3$. Under the limit $\phi_s \to \id$, the strips from $b_1$ to $a_1$ gives a multiple of $x$ and they exactly correspond to strips which give the $x^2$-term in $h_{\bar{X}}$. In the same way the strips $b_1 \mapsto a_2$ correspond to those which give $xy$-term of $h_{\bar{Z}}$, and $b_1 \mapsto a_3$ correspond to the $xz$-term in $h_{\bar{Y}}$. The signs of strips $b_1 \mapsto a_i$ and those of corresponding polygons coincide.

Summarizing, the matrix factorization $CF(\ltl,L)$ is given by

\[\xymatrixcolsep{2pc}\xymatrix{
\ar[r]&
\CAtl(-3) \oplus \CAtl(-4)^3  \ar[rrr]^-{\left(\begin{smallmatrix}0 & -x & -y & -z 
 \\-x & Q'_{22} & Q'_{23}& Q'_{24}
 \\-y & Q'_{32} & Q'_{33} & Q'_{34} \\-z &Q'_{42} & Q'_{43} & Q'_{44}\end{smallmatrix}\right)}
   & & &\CAtl(-3)\oplus \CAtl(-2)^3}\] 
\[\xymatrixcolsep{2pc}\xymatrix{  
 \ar[rrr]^-{\left(\begin{smallmatrix} 0& Q_{12}& Q_{13} & Q_{14} \\ Q_{21} &  \ctl x & \btl z & \atl y \\ Q_{31} & \atl z & \ctl y & \btl x \\ Q_{41} & \btl y &  \atl x & \ctl z\end{smallmatrix}\right)} 
 & & & \CAtl\oplus \CAtl(-1)^3 \ar[r] & \cdots}\]
 with quadratic entries $Q_{ij}$ and $Q'_{ij}$.

Similarly, $CF(\ltl,L')$ is computed as

\[\xymatrixcolsep{2pc}\xymatrix{
\ar[r] &
\CAtl(-2) \oplus \CAtl(-3)^3  \ar[rrr]^-{\left(\begin{smallmatrix}0 & -x & -y & -z 
\\ -x & Q'_{22} & Q'_{23}& Q'_{24}
 \\-y & Q'_{32} & Q'_{33} & Q'_{34} 
 \\-z &Q'_{42} & Q'_{43} & Q'_{44}\end{smallmatrix}\right)} 
 & && \CAtl(-2)\oplus \CAtl(-1)^3 }\]
 \[\xymatrix{\ar[rrr]^-{\left(\begin{smallmatrix} 0& Q_{12}& Q_{13} & Q_{14} \\ Q_{21} &  \ctl x & \btl z & \atl y \\ Q_{31} & \atl z & \ctl y & \btl x \\ Q_{41} & \btl y &  \atl x & \ctl z\end{smallmatrix}\right)} 
 & & &
 \CAtl(1)\oplus \CAtl^3 \ar[r]& \cdots .}\]
Now it is clear that these matrix factorizations coincide with above stabilizations of $\bk$ and $\CBtl(1)_{\geq 0}$ respectively as objects of $D^{gr}_{sg}(\CBtl)$, because the quadratic entries contribute as a homotopy $s_1$ in (\ref{whomotopy}).

Morphisms between $L_{(1,0)}$ and $L_{(1,-3)}$ are mapped via $\locmir^\ltl \circ \CS_0$ as follows. It suffices to compute area $0$(after the limit $t \to 0$) holomorphic triangles because they completely classify representatives of morphisms of matrix factorizations. We count area $0$ triangles in Figure \ref{landlprime} whose boundaries are schematically as Figure \ref{schematic}. 

\begin{figure}
\includegraphics[height=2.3in]{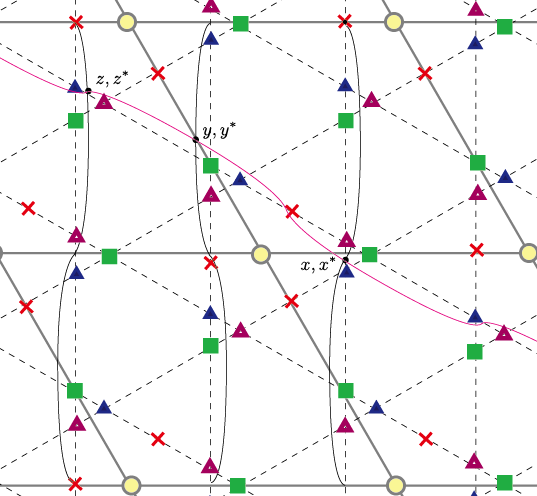}
\caption{Intersections between $\CS_0(L)$ and $\CS_0(L')$.}
\label{landlprime}
\end{figure}

\begin{figure}
\includegraphics[height=1.3in]{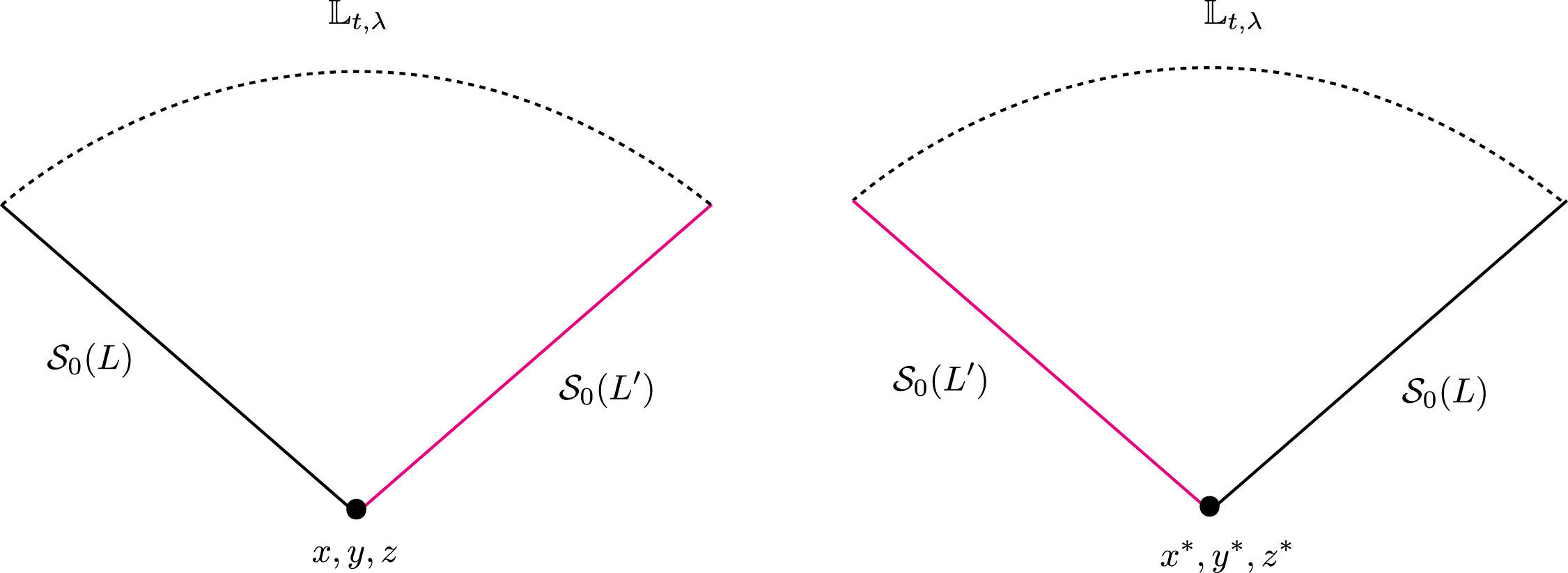}
\caption{The left one expresses degree 0 morphisms from $\locmir^\ltl(\CS_0(L))$ to $\locmir^\ltl(\CS_0(L'))$, and the right one corresponds to degree 1 morphisms the other way.}
\label{schematic}
\end{figure}

The first conclusion is that degree 0 morphisms from $\locmir^\ltl(\CS_0(L))$ to $\locmir^\ltl(\CS_0(L'))$ are given by the following:

\[\xymatrixcolsep{2pc}\xymatrix{
\ar[r]&
\CAtl(-3) \oplus \CAtl(-4)^3 \ar[dd]_{p^{-2},q^{-2},r^{-2}} \ar[rrr]^-{\left(\begin{smallmatrix}0 & -x & -y & -z 
 \\-x & Q'_{22} & Q'_{23}& Q'_{24}
 \\-y & Q'_{32} & Q'_{33} & Q'_{34} \\-z &Q'_{42} & Q'_{43} & Q'_{44}\end{smallmatrix}\right)}
   & & &\CAtl(-3)\oplus \CAtl(-2)^3 \ar[dd]^{p^{-1},q^{-1},r^{-1}} \ar[r]
 & \cdots \\ \\ 
 \ar[r] &
\CAtl(-2) \oplus \CAtl(-3)^3  \ar[rrr]^-{\left(\begin{smallmatrix}0 & -x & -y & -z 
\\ -x & Q'_{22} & Q'_{23}& Q'_{24}
 \\-y & Q'_{32} & Q'_{33} & Q'_{34} 
 \\-z &Q'_{42} & Q'_{43} & Q'_{44}\end{smallmatrix}\right)} 
 & && \CAtl(-2)\oplus \CAtl(-1)^3 \ar[r]
 &
 \cdots }\]
 where the intersection $x$ induces
 \[ p^{2i}= \left(\begin{array}{cccc} 0& * & * & * \\ 1 & * & * & * \\ 0 & * & * & * \\ 0 & * & * & * \end{array}\right), \;p^{2i-1}=\left(\begin{array}{cccc} 0 & 1 & 0 & 0 \\ * & * & * & * \\ * & * & * & * \\ * & * & * & * \end{array}\right)\]
 and $y$, $z$ induce morphisms $q$ and $r$ respectively as following:
 \[ q^{2i}=\left(\begin{array}{cccc} 0 & * & * & * \\ 0 & * & * & * \\ 1 & * & * & * \\ 0 & * & * & * \end{array}\right), \; q^{2i-1}=\left(\begin{array}{cccc} 0 & 0 & 1 & 0 \\ * & * & * & * \\ * & * & * & * \\ * & * & * & * \end{array}\right),\]
 \[r^{2i}=\left(\begin{array}{cccc} 0 & * & * & * \\ 0 & * & * & * \\ 0 & * & * & * \\ 1 & * & * & * \end{array}\right), \; r^{2i-1}=\left(\begin{array}{cccc} 0 & 0 & 0 & 1 \\ * & * & * & * \\ * & * & * & * \\ * & * & * & * \end{array}\right).\]
 
The same computation gives degree 1 morphisms from $\locmir^\ltl(\CS_0(L'))$ to $\locmir^\ltl(\CS_0(L))$ as follows.
 
\[\xymatrixcolsep{2pc}\xymatrix{
 \ar[r] &
\CAtl(-2) \oplus \CAtl(-3)^3 \ar[dd]_{p'^{-2},q'^{-2},r'^{-2}} \ar[rrr]^-{\left(\begin{smallmatrix}0 & -x & -y & -z 
\\ -x & Q'_{22} & Q'_{23}& Q'_{24}
 \\-y & Q'_{32} & Q'_{33} & Q'_{34} 
 \\-z &Q'_{42} & Q'_{43} & Q'_{44}\end{smallmatrix}\right)} 
 & && \CAtl(-2)\oplus \CAtl(-1)^3 \ar[dd]^{p'^{-1},q'^{-1},r'^{-1}} \ar[r]
  &
 \cdots \\ \\ 
\ar[r]&
\CAtl(-3) \oplus \CAtl(-2)^3  \ar[rrr]^-{\left(\begin{smallmatrix} 0& Q_{12}& Q_{13} & Q_{14} \\ Q_{21} &  \ctl x & \btl z & \atl y \\ Q_{31} & \atl z & \ctl y & \btl x \\ Q_{41} & \btl y &  \atl x & \ctl z\end{smallmatrix}\right)}
   & & &\CAtl(-3)\oplus \CAtl(-2)^3 \ar[r]
 &\cdots
  }\] 
The intersections $x^*$, $y^*$ and $z^*$ induce morphisms by $p'$, $q'$ and $r'$ respectively and they are represented by
\[ p'^{2i}=\left(\begin{array}{cccc}0 & 1 & 0 & 0 \\1 & * & * & * \\0 & * & * & * \\0 & * & * & *\end{array}\right),\;q'^{2i}=\left(\begin{array}{cccc}0 & 0 & 1 & 0 \\0 & * & * & * \\1 & * & * & * \\0 &  * & * & *\end{array}\right),\; r'^{2i}=\left(\begin{array}{cccc}0 & 0 & 0 & 1 \\0 & * & * & * \\0 & * & * & * \\1 &  * & * & *\end{array}\right).\]

It is now straightforward that these morphisms coincide with previous one constructed from Polishchuk-Zaslow's mirror functor followed by Orlov's LG/CY functor.

The description of $\CS_i$ for all $i\in \ZZ$ is only slightly different from the commutative case, which was dealt with in \cite{L}. Let $j=\lfloor -\frac{i}{3} \rfloor$, $d=-i-3j$, so that $d=0$, $1$ or $2$. Recall the notation $\tau$ for $\ZZ/3$-action, and let $S_i$ be a symplectomorphism $\tau^d\circ \phi_i$, where $\phi_i$ is a symplectomorphism which maps the graded Lagrangian $L_{(1,3i)}$ to the graded Lagrangian $\LL^{vert}_{t,\lambda}$, and the graded Lagrangian $L_{(1,3i-3)}$ to the graded Lagrangian $\tau(\LL^{vert}_{t,\lambda})$. Now, the functor $\CS_i$ is given by
\[ \CS_i= [-j]\circ S_i\]
and its proof is parallel to the commutative case in \cite{L}.
\end{proof}

\bibliographystyle{amsalpha}

\end{document}